\documentclass[12pt]{amsart}

\usepackage{amssymb,geometry,stmaryrd,color}
\usepackage{amsmath}
\usepackage{amsfonts}
\usepackage{fullpage}

\usepackage{hyperref}
\hypersetup{colorlinks,linkcolor={red},citecolor={blue},urlcolor={blue}}

\usepackage{graphicx}
\usepackage{caption}
\begin{document}

\newtheorem{theorem}{Theorem}[section]
\newtheorem{lemma}[theorem]{Lemma}
\newtheorem{proposition}[theorem]{Proposition}
\newtheorem{corollary}[theorem]{Corollary}
\newtheorem{conjecture}[theorem]{Conjecture}
\newtheorem{example}{Example}

\newtheorem{definition}[theorem]{Definition}
\newtheorem{remark}[theorem]{Remark}
\newtheorem{notation}[theorem]{Notation}
\newtheorem{question}[theorem]{Question}

\numberwithin{equation}{section}

\def\s{{\bf s}} 
\def\t{{\bf t}} 
\def\u{{\bf u}} 
\def\x{{\bf x}} 
\def\y{{\bf y}} 
\def\z{{\bf z}} 
\def\B{{\bf B}} 
\def\C{{\bf C}} 
\def\D{{\bf D}}
\def\K{{\bf K}}
\def\F{{\bf F}}
\def\M{{\bf M}}
\def\ML{{\bf ML}}
\def\Nn{{\bf N}}
\def\G{{\bf \Gamma}} 
\def\W{{\bf W}}
\def\X{{\bf X}}
\def\U{{\bf U}}
\def\V{{\bf V}}
\def\Un{{\bf 1}}
\def\Y{{\bf Y}}
\def\Z{{\bf Z}}
\def\P{{\bf P}}
\def\Q{{\bf Q}}
\def\S{{\bf S}}
\def\L{{\bf L}}
\def\T{{\bf T}}

\def\cB{{\mathcal{B}}} 
\def\cC{{\mathcal{C}}} 
\def\cD{{\mathcal{D}}} 
\def\cG{{\mathcal{G}}} 
\def\cK{{\mathcal{K}}} 
\def\cL{{\mathcal{L}}} 
\def\cR{{\mathcal{R}}} 
\def\cS{{\mathcal{S}}}
\def\cU{{\mathcal{U}}}
\def\cV{{\mathcal{V}}} 
\def\cX{{\mathcal X}}
\def\cY{{\mathcal Y}}
\def\cZ{{\mathcal Z}}

\def\Ea{E_\a}
\def\eps{{\varepsilon}} 
\def\esp{{\mathbb{E}}} 
\def\Ga{{\Gamma}}

\def\lacc{\left\{}
\def\lcr{\left[}
\def\lpa{\left(}
\def\lva{\left|}
\def\racc{\right\}}
\def\rpa{\right)}
\def\rcr{\right]}
\def\rva{\right|}

\def\prst{{\leq_{st}}}
\def\prost{{\prec_{st}}}
\def\prcvx{{\prec_{cx}}}
\def\Rr{{\bf R}}

\def\CC{{\mathbb{C}}}
\def\EE{{\mathbb{E}}}
\def\NN{{\mathbb{N}}} 
\def\QQ{{\mathbb{Q}}} 
\def\PP{{\mathbb{P}}}
\def\ZZ{{\mathbb{Z}}}
\def\RR{{\mathbb{R}}}

\def\Tt{{\bf \Theta}}
\def\Ttt{{\tilde \Tt}}

\def\a{\alpha}
\def\A{{\bf A}}
\def\AA{{\mathcal A}}
\def\hAA{{\hat \AA}}
\def\hL{{\hat L}}
\def\hT{{\hat T}}

\def\claw{\stackrel{(d)}{\longrightarrow}}
\def\elaw{\stackrel{(d)}{=}}
\def\pslaw{\stackrel{a.s.}{\longrightarrow}}
\def\qed{\hfill$\square$}

\newcommand*\pFqskip{8mu}
\catcode`,\active
\newcommand*\pFq{\begingroup
        \catcode`\,\active
        \def ,{\mskip\pFqskip\relax}%
        \dopFq
}
\catcode`\,12
\def\dopFq#1#2#3#4#5{%
        {}_{#1}F_{#2}\biggl[\genfrac..{0pt}{}{#3}{#4};#5\biggr]%
        \endgroup
}

\def\ii{{\rm i}}

\title[ID of alpha Cauchy]{Infinite divisibility of $\alpha$-Cauchy distributions}

\author[M.~Wang]{Min Wang}

\address{School of Mathematics and Statistics, Wuhan University of Technology,  Wuhan, 430063, China}

\email{minwangmath@whut.edu.cn}



\begin{abstract} 
   In 2009, Yano, Yano and Yor proposed the question of studying the infinite divisibility of  the $\alpha$-Cauchy variable $\mathcal{C}_\alpha$ for $\alpha > 1$. The particular case $\mathcal{C}_2$ is the well-known standard Cauchy variable, which is infinitely divisible and indeed stable. For $\alpha \neq 2$, the infinite divisibility of $\mathcal{C}_\alpha$ is previously unknown. In this paper, we prove that $\mathcal{C}_\alpha$ is infinitely divisible if and only if $1 < \alpha \leq 2$. 
\end{abstract}

\maketitle

\section{Introduction}
A random variable is infinitely divisible if, for any positive integer $n$, it is the sum of $n$ independent and identically distributed random variables. Typical examples of infinitely divisible distributions include the normal, Cauchy, half Cauchy, Stable, Gamma, Poisson, and Student-$t$. In 1929, de Finetti first introduced the concept of infinite divisibility, and then Kolmogorov, Lévy, and Khintchine further developed the theory. These distributions are closely related to limit theorems and the theory of Lévy processes and have important applications in finance, insurance, biology, physics, and signal processing. 

The study of infinite divisibility of random variables is a long-standing research topic.
Steutel \cite{Ste70, Ste73}, Hudson-Tucker \cite{HT75}, Kristiansen \cite{Kri94}, Sato \cite{Sat99}, Steutel-Van Harn \cite{SH04}, and others have provided a series of criteria for determining infinite divisibility from aspects such as density, characteristic function, Laplace transform, tail behaviour and decomposition. 
In some cases, it is more convenient to consider subclasses of infinitely divisible distributions, like self-decomposable distributions, generalized gamma convolutions and hyperbolically completely monotone distributions. We refer to Steutel and Van Harn \cite{SH04}, Sato \cite{Sat99} and Bondesson \cite{Bon92} for abundant properties of these subclasses. 
Proving or disproving infinite divisibility of a certain distribution can sometimes be quite sophisticated. For example, the infinite divisibility of the Student $t$-distribution \cite{Gro76a, Ism77}, the Pareto distribution \cite{Tho77a}, the lognormal distribution \cite{Tho77b}, the inverse Beta distributions \cite{BS15}, and others \cite{JS13, BLM20}.

In 2009, Yano, Yano and Yor \cite{YYY09} introduced the $\alpha$-Cauchy variable to study the first hitting times of points for one-dimensional symmetric stable L\'evy processes.  
The density of $\alpha$-Cauchy variable, denoted by $\mathcal{C}_\a$, is 
\begin{equation}\label{eq density alpha Cauchy}
    f_{\mathcal{C}_\a}(x) = \frac{\sin(\pi/\a)}{2\pi/\a}\frac{1}{1+|x|^\a}, \quad \alpha > 1, \,\, x \in \mathbb{R}. 
\end{equation}
In particular, $\mathcal{C}_2$ is 
the standard Cauchy random variable, which is the ratio of two independent normal random variables. It is a continuous distribution describing resonance behavior. 
Cauchy distribution and $\alpha$-Cauchy distribution have been used to calculate some special values of the Riemann zeta function and Hurwitz zeta function, respectively; see Bourgade-Fujita-Yor \cite{BFY07} and Fujita-Yano \cite{FY12}.

Yano, Yano and Yor \cite[Remark 2.9]{YYY09} proposed questions of studying the infinite divisibility and self-decomposability of  $\mathcal{C}_\alpha$, $|\mathcal{C}_\alpha|$, and $|\mathcal{C}_\alpha|^{-p}$ for $ p>0$ and $\alpha > 1$. 
In this paper, we proceed with the first step: to study the infinite divisibility of $\mathcal{C}_\a$. Other questions will be left for future research. 

It is well-known that the Cauchy variable $\mathcal{C}_2$ is infinitely divisible and indeed stable. For $\alpha \neq 2$, the infinite divisibility of $\mathcal{C}_\alpha$ is previously unknown. In the following theorem, we provide the complete answer.

\begin{theorem}\label{Theorem ID of alpha Cauchy}
The $\alpha$-Cauchy variable $\mathcal{C}_\a$ is infinitely divisible if and only if  $1 < \a \leq 2$. 
\end{theorem}

We will prove the sufficient part of this theorem using three distinct methods. In doing so, we will uncover several key properties of $\mathcal{C}_\alpha$, including novel decompositions and new expressions for its density and characteristic function. To prove necessity, we will show that the characteristic function of $\mathcal{C}_\alpha$ can take negative values when $\alpha > 2$, which is impossible for an infinitely divisible random variable.

Let $X_\alpha = (X_\alpha(t): t\geq 0)$ be the symmetric stable L\'evy process of index $\alpha$ starting from zero. 
Let $T_{\{a\}}(X_\alpha)$ denote the first hitting time of point $a$ for $X_\alpha$.
The infinite divisibility of  $T_{\{a\}}(X_\alpha)$ implies that of $\mathcal{C}_\a$, see formula (5.12) in \cite{YYY09}. 
In 2019, Letemplier and Simon \cite{LS19} further studied the first hitting time of zero $\tau(\alpha, \rho)$ for a real strictly $\alpha$-stable process (can be asymmetric) starting from one. They conjectured that $\tau(\alpha, \rho)$ is infinitely divisible for $\alpha \in (1,2]$ and $\rho \in [0,1]$. Note that $\tau(\alpha, \frac{1}{2})$ coincides with $T_{\{-1\}}(X_\alpha)$. Our Theorem \ref{Theorem ID of alpha Cauchy} supports Letemplier and Simon's conjecture.

\textit{Notation.} Throughout, in any factorization of the type $X \elaw Y \times Z$, 
the random variables $Y,Z$ on the right-hand side will be assumed to be independent. We denote the Gamma random variable by $\G_{c}$, whose density is
$ \frac{1}{\Gamma(c)}x^{c-1}e^{-x}\mathbf{1}_{(0, \infty)}(x). $

\section{Proof of the sufficient part of Theorem \ref{Theorem ID of alpha Cauchy} }
\label{section sufficient}
We give three proofs from three aspects: decomposition, density, and characteristic function.
These proofs are based on three key properties of $\mathcal{C}_\alpha$, see Lemmas \ref{lemma factorization alpha cauchy}, \ref{lemma density as LT of a CM} and \ref{lemma density of Linnik}. 
\subsection{From the perspective of decomposition:} 
By Bochner's subordination, see \cite{Boc49} or \cite[Theorem 30.1]{Sat99}, one knows that a variance mixture of normal distribution functions is infinitely divisible if the mixing distribution function is infinitely divisible, i.e. for $Z \, \sim \, N(0,1)$, $Z \times \sqrt{T}$ is infinitely divisible if T is infinitely divisible. There are currently no known necessary and sufficient conditions for $T$ such that  $Z \times \sqrt{T}$ is infinitely divisible. 
In 1971, Kelker \cite{Kel71} has shown that even if T is not infinitely divisible, $Z \times \sqrt{T}$ can still be infinitely divisible. His result can be reformulated as follows.

\begin{lemma}
    \label{Kelker Theorem 5}
    Let $Z$ be a standard normal random variable and let $W$ be a nonnegative random variable. Then the random variable $Z \times \sqrt{\G_{1/2}^{-1} \times W }$ is infinitely divisible. 
\end{lemma}

\begin{proof}[Proof of Lemma \ref{Kelker Theorem 5}]
Theorem 5 in \cite{Kel71} stated that \textit{all scale parameter mixtures of Cauchy distributions are infinitely divisible}. A mixture of Cauchy distributions is also a mixture of zero-mean normal distributions, because Cauchy variable is proportional to $Z \times \sqrt{\G_{1/2}^{-1}}$. Therefore, $Z \times \sqrt{\G_{1/2}^{-1} \times W }$ is infinitely divisible for any nonnegative random variable $W$. 
\end{proof}

The fractional moments of the half $\alpha$-Cauchy variable $ |\mathcal{C}_\a|$ are of the form
\begin{equation}
\label{eq momemts of alpha cauchy}
    \mathbb{E}\left[|\mathcal{C}_\a|^s\right] =  \frac{\sin(\pi/\a)}{\pi} \Gamma\left(\frac{1}{\a} +\frac{s}{\a}\right)\Gamma\left(1 - \frac{1}{\a} - \frac{s}{\a}\right), \quad -1 < s < \a-1.
\end{equation}
We want to prove that $\mathcal{C}_\a$ has a factorization in the form of Lemma \ref{Kelker Theorem 5} by comparing fractional moments of both sides.  

\begin{lemma}
    \label{lemma factorization alpha cauchy}
    If $1<\alpha \leq 2$, then the $\alpha$-Cauchy variable is proportional to 
    $$ Z \times \sqrt{\G_{1/2}^{-1} \times T_\a},$$
    where $T_\alpha$ is a positive random variable that satisfies
    \begin{equation}
  \label{eq mellin T alpha}
    \mathbb{E}[T_{\a}^s] = \frac{\Gamma(\frac{1}{2})\Gamma(\frac{1}{2})}{\Gamma(\frac{1}{\alpha})\Gamma(1-\frac{1}{\alpha})} \frac{\Gamma(\frac{1}{\a}+\frac{2s}{\a})\Gamma(1-\frac{1}{\a}-\frac{2s}{\a})}{\Gamma(\frac{1}{2}+s)\Gamma(\frac{1}{2}-s)}, \quad -\frac{1}{2} < s < \frac{\alpha-1}{2}.
\end{equation}
\end{lemma}

\begin{proof}[Proof of Lemma \ref{lemma factorization alpha cauchy}]
    Ferreira and Simon \cite{FS23} proved, for real $a$ and $b$, the random variable $\M_{a,b}$ with fractional moments 
\begin{equation}
\label{eq mellin M alpha beta}
    \mathbb{E}[\M_{a,b}^s] = \Gamma(a+b)\frac{\Gamma(1+s)}{\Gamma(a+b+a s)}, \quad s > -1,
\end{equation}
exists if and only if $a \in [0,1]$ and $ b \geq 0$. Note that $\M_{1,b}$ is the Beta random variable. For $a \in [0, 1)$ and $ b \geq 0$, the density of $\M_{a,b}$ is $\Gamma(a +b) \phi(-a, b, -x)$, where 
$$ \phi(a, b, z) = \sum_{n\geq 0} \frac{z^n}{n! \Gamma(b+a n) }, \quad b, z \in \mathbb{C}, \,\, a > -1, $$
is the Wright function. 
When $t > -1, a \in [0,1)$ and $ b \geq 0$, the random variable with density $\frac{\Gamma(a(1+t)+b)}{\Gamma(1+t)} x^{t} \phi(-a, b, -x)$ exists, and we denote it by $\M_{a,b, t}$. By direct computations, we have
\begin{equation}
\label{eq mellin M alpha beta t}
    \mathbb{E}[\M_{a,b, t}^s] = \frac{\Gamma(a(1+t)+b)}{\Gamma(1+t)}\frac{\Gamma(1+t+s)}{\Gamma(a(1+t)+b+a s)}, \quad s > -1-t. 
\end{equation}
By \eqref{eq momemts of alpha cauchy}, \eqref{eq mellin T alpha}, \eqref{eq mellin M alpha beta t} and comparing fractional moments, we can check that the desired $T_\alpha$ can be chosen as 
 $\M_{\frac{\alpha}{2},0, \frac{1}{\alpha} -1}^{2/\alpha} \times \M_{\frac{\alpha}{2},1-\frac{\alpha}{2}, -\frac{1}{\alpha} }^{-2/\alpha} $. 
\end{proof}

\begin{proof}[First proof of the sufficient part of Theorem \ref{Theorem ID of alpha Cauchy}]
    Combining Lemmas \ref{Kelker Theorem 5} and \ref{lemma factorization alpha cauchy}, the  $\alpha$-Cauchy variable is infinitely divisible if $1<\alpha \leq 2$. 
\end{proof}

\subsection{From the perspective of density:} 
Recall that a non-negative function $f$ on $(0, \infty)$ is called completely
monotone if it has derivatives of all orders and
$ (-1)^n f^{(n)}(x) \geq 0$
 for $n \geq 1$
and $x > 0$, see Schilling-Song-Vondraček \cite{SSV12}.  
Kelker \cite{Kel71} also found an interesting duality between density and characteristic function. 
\begin{theorem}\label{Theorem6 in kelker1971}{\cite[Theorem 6]{Kel71}}
Let $h: (-\infty, \infty) \to (0, \infty)$ and $g: (0,\infty) \to (0, \infty)$ be two functions that satisfy $h(0) = 1$, 
$h(t) = \int_0^\infty e^{-t^2 u}g(u)du$ and $\int_{-\infty}^{\infty}h(t)dt = K < \infty$. If $g$ or $u^{-3/2}g(1/u)$ is completely monotone, then $h(t)$ is an infinitely divisible characteristic function and $K^{-1}h(x)$ is an infinitely divisible density. 
\end{theorem}
This theorem gives us a way to verify infinite divisibility from the perspective of density. Generally, it is difficult to prove the complete monotonicity of the kernel function. In our case, we will take advantage of the three-parametric Mittag-Leffler function to overcome this difficulty.

The classical Mittag-Leffler function is the entire function 
\begin{equation*}
    E_{\rho}(z) := \sum_{n\geq 0} \frac{z^n}{\Gamma(1+\rho n)}, \quad z \in \mathbb{C},\, \rho >0.
\end{equation*}
It was introduced by Gosta Mittag-Leffler in 1903. Two years later, Wiman introduced the two-parametric Mittag-Leffler function
\begin{equation*}
    E_{\rho, \mu}(z) := \sum_{n\geq 0} \frac{z^n}{\Gamma(\mu +\rho n)}, \quad z \in \mathbb{C},\, \rho, \mu >0.
\end{equation*}
In 1971, Prabhakar studied the three-parametric Mittag-Leffler function
\begin{equation*}
    E_{\rho, \mu}^\gamma (z) := \sum_{n\geq 0}   \frac{\Gamma(\gamma + n)}{\Gamma(\gamma)\Gamma(1 + n)\Gamma(\mu +\rho n)} z^n, \quad z \in \mathbb{C},
\end{equation*}
where $\rho >0$, $\mu > 0$ and $\gamma > 0$.
With the help of these functions, we can express $f_{\mathcal{C}_\a}(\sqrt{|x|})$ as the Laplace transform of a completely monotone function when $1< \alpha \leq 2$.

\begin{lemma}\label{lemma density as LT of a CM}
   For any $\alpha > 1$, the density of $\mathcal{C}_\a$ can be transformed into
\begin{equation}
\label{transformed density}
   f_{\mathcal{C}_\alpha}(x) = \frac{\sin(\pi/\a)}{2\pi/\a} \int_0^{\infty} e^{-x^2 t}  t^{\alpha/2-1} E_{\alpha/2, \alpha/2}(-t^{\alpha/2})dt, \quad \alpha > 1.
\end{equation}
Moreover, if $1< \alpha \leq 2$, then the function $ t^{\alpha/2-1} E_{\alpha/2, \alpha/2}(-t^{\alpha/2})$ is completely monotone. 
\end{lemma}

\begin{proof}[Proof of Lemma \ref{lemma density as LT of a CM}]
For $s > 1$ and $\rho > 0$, by direct computation, 
\begin{align*}
    & \int_0^\infty e^{-st} t^{\rho-1}E_{\rho,\rho}(-t^\rho)dt \\
    = & \sum_{n\geq 0} \frac{(-1)^n}{\Gamma(\rho + \rho n )}\int_0^\infty e^{-st} t^{\rho-1+n \rho }dt\\
     =&  \sum_{n\geq 0} \frac{(-1)^n}{s^{\rho + n\rho}}  = \frac{1}{1+s^\rho}.
\end{align*}
By analytic continuation, the identity 
\begin{align*}
     \int_0^\infty e^{-st} t^{\rho-1}E_{\rho,\rho}(-t^\rho)dt   = \frac{1}{1+s^\rho}
\end{align*}
holds true for $s > 0$ and $\rho > 0$. 
 Choosing $s = x^2, $ and $ \rho = \alpha/2$, we have 
\begin{equation}
    \frac{1}{1+|x|^\alpha} = \frac{1}{1+(x^2)^{\alpha/2}} =  \int_0^{\infty} e^{-x^2 t}  t^{\alpha/2-1} E_{\alpha/2, \alpha/2}(-t^{\alpha/2})dt, \quad \alpha > 1.
\end{equation}
If $1< \alpha \leq 2$, then $t^{\alpha/2-1} E_{\alpha/2, \alpha/2}(-t^{\alpha/2})$ is completely monotone, because
\begin{equation*}
     t^{\mu-1} E^\gamma_{\rho, \mu}(-t^\rho) \; \text{is completely monotone iff \, $0<\rho, \mu \leq 1$ \, and \, $0 < \gamma \leq\mu/\rho$, }
\end{equation*}
see, e.g. \cite[formula 5.1.10]{GKMR20} or more recently \cite{GHLP21}.
\end{proof}

\begin{remark}
The above lemma proves that 
the function $x \mapsto f_{\mathcal{C}_\a}(\sqrt{|x|})$ is the Laplace transform of a completely monotone function when $1< \alpha \leq 2$. In other words, it is a Stieltjes function. 
This conclusion can also be proved by utilizing the relationship between Stieltjes functions and complete Bernstein functions: a function $f \not\equiv 0 $ is a complete Bernstein function, if and only if, $1/f \not\equiv 0$ is a Stieltjes function, see the monograph of Schilling, Song and Vondraček \cite[Theorem 7.3]{SSV12}. For every $\beta \in (0,1)$, the function $x \mapsto 1 + x^\beta$ is a complete Bernstein function, because
    \begin{equation*}
      1 + x^\beta = 1 + \frac{\beta}{\Gamma(1-\beta)}\int_0^\infty(1-e^{-xt})\frac{1}{t^{\beta+1}}dt, \quad \beta \in (0,1). 
    \end{equation*}
This proof is simple, but Lemma \ref{lemma density as LT of a CM} gives an explicit expression that may have independent interest. 
\end{remark}

\begin{proof}[Second proof of the sufficient part of Theorem \ref{Theorem ID of alpha Cauchy}]
    Combining Theorem \ref{Theorem6 in kelker1971} and Lemma \ref{lemma density as LT of a CM}, $\mathcal{C}_\a$ is infinitely divisible if  $1 < \a \leq 2$. 
\end{proof}

\subsection{From the perspective of characteristic function:}
The criterion we will use is the following theorem.
\begin{theorem}{\cite[Theorem IV.10.5]{SH04}}\label{Steutel 10.5}
    If $\pi$ is a Laplace transform of a density function, then the function $\phi$ defined by $\phi(u) = \pi(|u|)$ for $u \in \mathbb{R}$, is an infinitely divisible characteristic function. 
\end{theorem}
Recall that for $0 < \alpha \leq 2$, the $\alpha$-Linnik variable $\Lambda_\alpha$ is defined by its characteristic function:
\begin{equation*}
    \mathbb{E}[e^{i\theta \Lambda_\alpha}] = \frac{1}{1+|\theta|^\alpha}, \; \theta \in \mathbb{R}.
\end{equation*}
Devroye \cite{Dev90} proved the following identity in law 
$$   \Lambda_\alpha  \elaw X_\alpha(\G_1), $$
where $X_\alpha = (X_\alpha(t): t\geq 0)$ is the symmetric stable L\'evy process of index $\alpha$ starting from zero as before and $\G_1$ is a standard exponential variable independent of $X_\alpha$. 

By \cite[Proposition 2.11]{YYY09}, for $1<\alpha<2$, the $\alpha$-Cauchy distribution and the $\alpha$-Linnik distribution satisfy the following relation:
``\textit{the characteristic function of any of these two distributions is proportional to the density of the other}".

We then focus on the density of $\alpha$-Linnik variable in order to express the characteristic function of $\alpha$-Cauchy variable. 
\begin{lemma}
    \label{lemma density of Linnik}
    For $1<\alpha <2$, the density of $\alpha$-Linnik variable, as well as the characteristic function of $\alpha$-Cauchy variable, is proportional to
$$  \int_0^\infty  e^{-|x|y}  \frac{ y^\alpha}{  y^{2\alpha} + 2\cos(\pi \alpha/2)y^\alpha + 1}dy. $$
\end{lemma}

\begin{proof}[Proof of Lemma \ref{lemma density of Linnik}]

Zolotarev \cite[Theorem 2.6.3]{Zol86} proved that the fractional moment of $|X_\alpha(1)|$ is
\begin{equation*}\label{eq Mellin Z}
 \mathbb{E}[ |X_{\a}(1)|^s] = \frac{ \Gamma(1+s)\Gamma(1-s/\a)}{\Gamma(1+s/2)\Gamma(1-s/2)}   , \quad -1 < s < \a. 
\end{equation*}
By self-similarity, $   \Lambda_\alpha  \elaw X_\alpha(1) \times \G_1^{1/\alpha}$, thus the fractional moment of $|\Lambda_\alpha|$ is
\begin{equation}\label{eq Mellin Linnik}
 \mathbb{E}[ |\Lambda_{\a}|^s] = \frac{ \Gamma(1+s)\Gamma(1-s/\a) \Gamma(1+s/\alpha) }{\Gamma(1+s/2)\Gamma(1-s/2)} , \quad -1 < s < \a. 
\end{equation}
This yields a new identity in law: 
\begin{equation}
\label{new indentity in law}
     |\Lambda_\alpha|  \elaw \G_1 \times \Z_{\alpha/2}^{1/2}  \times \Z_{\alpha/2}^{-1/2},
\end{equation}
where $\Z_c, \, c \in (0,1),$ is a positive $c$-stable random variable, whose fractional moments are, see again \cite[Theorem 2.6.3]{Zol86}, 
\begin{equation}
    E[\Z_c^{s}] = \frac{\Gamma(1-s/c)}{\Gamma(1-s)}, \quad  s < c.
\end{equation}
Bosch \cite{Bos15} has studied the ratio of two independent positive $c$-stable random variables and obtained an explicit density, see Theorem 1.1 therein:
\begin{equation}
\label{eq Theorem in Bosch2015}
    \left(\frac{\Z_c}{\tilde{\Z}_c}\right)^c  \, \sim \, \frac{\sin(\pi c)}{\pi c (x^2 + 2\cos(\pi c)x + 1)}, \quad x >0, 
\end{equation}
where $\tilde{\Z}_c$ is an independent copy of $\Z_c$. Therefore, this lemma is a consequence of  \eqref{new indentity in law} and \eqref{eq Theorem in Bosch2015}.     
\end{proof}

\begin{proof}[Third proof of the sufficient part of Theorem \ref{Theorem ID of alpha Cauchy}]
    Combining Theorem \ref{Steutel 10.5} and Lemma \ref{lemma density of Linnik}, $\mathcal{C}_\a$ is infinitely divisible if  $1 < \a \leq 2$. 
\end{proof}

\section{Proof of the necessary part of Theorem \ref{Theorem ID of alpha Cauchy} }
\label{section necessary}
 We use the fact that an infinitely divisible characteristic function has no real zero to prove this part, see e.g. Lemma 7.5 in \cite{Sat99}. Let $\alpha > 2$. We suppose that $\mathcal{C}_\alpha$ is infinitely divisible. Its characteristic function $\varphi_\alpha(t)$ satisfies
\begin{equation}
\label{charcfunc}
    \varphi_\alpha(t) \propto \int_{-\infty}^\infty e^{itx}\frac{1}{1+|x|^\alpha}dx,
\end{equation}
where $ f \propto g $ means there exists a positive constant $c$ such that $f = cg$. 
Because $\varphi_\alpha(0) = 1$ and $\varphi_\alpha(t)$ has no real zero, we have $\varphi_\alpha(t) > 0$ for all $t \in \mathbb{R}$. We calculate the Laplace transform $L_\alpha(s)$ of $\varphi_\alpha(t){\bf 1}_{(0,\infty)}(t)$, for $s >0$, 
\begin{align}
\label{line1}
L_\alpha(s) &\propto \int_0^\infty e^{-st} \left(\int_{-\infty}^\infty e^{itx}\frac{1}{1+|x|^\alpha}dx\right)dt \\
     &\propto \text{Re}\left[  \int_0^\infty \left( \int_0^\infty e^{-(s-ix)t}dt\right)\frac{1}{1+x^\alpha}dx \right]  \\
     &\propto \text{Re}\left[  \int_0^\infty (s-ix)^{-1} \frac{1}{1+x^\alpha}dx\right] \\
      &\propto \int_0^\infty\frac{s}{s^2+x^2}\frac{1}{1+x^\alpha}dx\\
      \label{line5}
       &\propto \int_0^\infty\frac{1}{1+y^2}\frac{1}{1+(sy)^\alpha}dy
\end{align}
The second step is guaranteed by the dominated convergence theorem. 
In the third step we have used the identity
\begin{equation*}
    \int_0^\infty x^{\gamma-1}e^{-zx}dx = \Gamma(\gamma)z^{-\gamma}
\end{equation*}
for $\gamma > 0$ and $\text{Re}(z) > 0$. Let $s \downarrow 0$ in \eqref{line1} and \eqref{line5}, using the monotone 
convergence theorem, we have 
$$  \int_{-\infty}^\infty \varphi_\alpha(t) dt =2 \int_0^\infty \varphi_\alpha(t) dt < \infty. $$
Therefore, $\varphi_\alpha(t)$ is a density function on $\mathbb{R}$ (up to a multiplicative constant). \\
By Fourier inversion, from \eqref{charcfunc} we obtain
\begin{equation}
\label{FourierInversion}
 \phi_\alpha(x) :=   \frac{1}{1+|x|^\alpha}  \propto \int_{-\infty}^\infty e^{-itx}\varphi_\alpha(t)dt = \int_{-\infty}^\infty e^{itx}\varphi_\alpha(t)dt.
\end{equation}
It means that $\phi_\alpha$ is a characteristic function of a random variable $X$. We can check that $\phi_\alpha'(0) = \phi_\alpha''(0) = 0$, leading to $\mathbb{E}(X) = \mathbb{E}(X^2) = 0$, which is impossible. Therefore, the assumption that $\mathcal{C}_\alpha$ is infinitely divisible is wrong, we finish the proof.

\bigskip

\noindent
\textbf{Acknowledgements.} I am grateful to Profs. Thomas Simon and René Schilling for many insightful discussions and valuable comments. I would also like to thank Prof. Fuqing Gao for his valuable comments and suggestions.


\bibliographystyle{plain} 
\bibliography{MomentGammaML}
\end{document}